\documentclass[a4paper,11pt,twoside]{amsart}
\usepackage[latin1]{inputenc}
\usepackage[T1]{fontenc}
\usepackage{amsmath}
\usepackage{amsthm}
\usepackage{amssymb}
\usepackage[all]{xy}

\newtheorem{theorem}{Theorem}[section]
\newtheorem{lemma}[theorem]{Lemma}
\newtheorem{proposition}[theorem]{Proposition}

\theoremstyle{definition}

\theoremstyle{remark}
\newtheorem{remark}[theorem]{Remark}
\newtheorem{example}[theorem]{Example}

\DeclareMathOperator{\rank}{rank}
\DeclareMathOperator{\ch}{ch}

\DeclareMathOperator{\Ext}{Ext}

\newcommand{\bbC}{\mathbb{C}}
\newcommand{\calO}{\mathcal{O}}
\newcommand{\longto}{\longrightarrow}
\newcommand{\supp}{\mathrm{supp}}

\title{New Examples of Stable Bundles on Calabi-Yau Threefolds}

\author{Bj\"orn Andreas}
\address{Institut f\"ur Mathematik,
Freie Universit\"at Berlin, Arnimallee 3, 14195 Berlin, Germany}
\email{andreasb\char`\@math.fu-berlin.de}

\author{Norbert Hoffmann}
\address{Institut f\"ur Mathematik,
Freie Universit\"at Berlin, Arnimallee 3, 14195 Berlin, Germany}
\email{norbert.hoffmann\char`\@fu-berlin.de}

\thanks{This work was supported by the SFB/647 ``Space-Time-Matter. Analytic and Geometric Structures'' of the DFG (German Research Foundation). B. A. is supported by project DFG-SFB 647/A3, N. H. is supported by project DFG-SFB 647/A11.}

\begin{document}

\begin{abstract}\noindent
  In this paper we present a construction of stable bundles on Calabi-Yau threefolds using the method of bundle extensions.
  This construction applies to any given Calabi-Yau threefold with $h^{1,1}>1$. We give examples of stable bundles of rank 2 and 4 constructed out of pure geometric data of the given Calabi-Yau space.
  As an application, we find that some of these bundles satisfy the physical constraint imposed by heterotic string anomaly cancellation. 
\end{abstract}

\maketitle

\section{Introduction}
In a series of papers \cite{AC06, AC07, AHG} a class of stable vector bundles on elliptically and $K3$-fibered Calabi-Yau threefolds has been constructed using the method of bundle extensions. One motivation for the previous study came from string theory model building where it is of importance to obtain stable vector bundles with prescribed Chern classes. The main method used to construct stable bundles on elliptic fibrations is the spectral cover 
construction \cite{FMW, FMWIII, Don1} which also applies for $K3$-fibrations \cite{AHG}, however, the Chern classes of the corresponding bundles are often not flexible enough to satisfy 
various physical constraints. Therefore bundle extensions provide a way to modify the Chern classes
of given bundles and so allow more flexibility in physical model building (for other work in this direction see
\cite{DOPWII}-\cite{thom1}).

In this note we give a new method for constructing stable bundles via extensions. The method applies 
to any given Calabi-Yau threefold $X$ with $h^{1,1}(X)>1$. To illustrate the method, we construct stable rank two bundles out of given line bundles and rank four bundles using the tangent bundle and a line bundle. The advantage is here that we do not rely on any auxiliary bundle construction to provide us with input bundles. As an application, we show that some of these bundles give solutions to the basic second Chern class constraint which is imposed by the heterotic string anomaly equation. A further application will be discussed in a physical companion paper \cite{AHSM11}.

The paper is organized as follows. In Section 2 we prove our basic result  Proposition 2.2 which 
provides conditions for constructing stable extension bundles on smooth projective varieties $X$ of 
complex dimension $n\geq 2$. In Section 3 we restrict $X$ to be a Calabi-Yau threefold and discuss the solvability of the conditions provided by Proposition 2.2 on these spaces. Moreover, we derive  
an explicit nonsplit condition via the Hirzebruch-Riemann-Roch Theorem. In Section 4 we restrict the construction to bundles of rank 2 and 4 and give explicit examples in the case of Calabi-Yau threefolds which admit an elliptic fibration, or a $K3$-fibration. We also analyze if some
of these bundles provide solutions to the heterotic string anomaly equation.

\section{Stable Extensions on a Smooth Projective Variety}

Let $X$ be a smooth projective variety of dimension $n \geq 2$ over $\bbC$.
We recall a standard boundedness property; cf. for example \cite{FMWIII}
\begin{lemma}
  Let $D_1, \ldots, D_{n-1}$ be ample divisor classes on $X$. Let $E$ be a torsionfree coherent sheaf of rank $r$ over $X$. Then there is a constant $C$ such that
  \begin{equation*}
    c_1( F) \cdot D_1 \cdot \ldots \cdot D_{n-1} \leq C
  \end{equation*}
  for all coherent subsheaves $F \subseteq E$.
\end{lemma}
\begin{proof}
  We choose a chain of coherent subsheaves
  \begin{equation*}
    0 = E_0 \subset E_1 \subset \ldots \subset E_r = E
  \end{equation*}
  such that $E_i$ has rank $i$ and $E_i/E_{i-1}$ is torsionfree for all $i$. Given a coherent subsheaf $F \subseteq E$, we put $F_i := F \cap E_i$.
  This defines a chain of coherent subsheaves
  \begin{equation*}
    0 = F_0 \subseteq F_1 \subseteq \ldots \subseteq F_r = F
  \end{equation*}
  such that $F_i/F_{i-1}$ is a subsheaf of $E_i/E_{i-1}$ for all $i$. Let $S$ be the set of indices $i$ for which $F_i/F_{i-1}$ is nonzero.
  Then $F_i/F_{i-1} \subseteq E_i/E_{i-1}$ are both torsionfree of rank $1$, so the divisor class $c_1( E_i/E_{i-1}) - c_1( F_i/F_{i-1})$ is effective.
  Taking the sum over all indices $i \in S$, we conclude that the divisor class
  \begin{equation*}
    \sum_{i \in S} c_1(  E_i/E_{i-1}) - c_1( F)
  \end{equation*}
  is effective. Since $D_1, \ldots, D_{n-1}$ are ample, $D \cdot D_1 \cdot \ldots \cdot D_{n-1}$ is nonnegative for every effective divisor class $D$ on $X$, so
  \begin{equation*}
    c_1( F) \cdot D_1 \cdot \ldots \cdot D_{n-1} \leq \sum_{i \in S} c_1(  E_i/E_{i-1}) \cdot D_1 \cdot \ldots \cdot D_{n-1} 
  \end{equation*}
  follows. Since there are only finitely many possibilities for the subset $S \subseteq \{1, \ldots, r\}$, the right hand side is bounded from above by some constant $C$.
\end{proof}
We recall that the \emph{slope} of a nonzero torsionfree coherent sheaf $E$ over $X$ with respect to an ample divisor class $H$ on $X$ is defined by
\begin{equation*}
  \mu_{H}( E) := \frac{c_1( E) \cdot H^{n-1}}{\rank( E)}.
\end{equation*}
The sheaf $E$ is called slope $H$-\emph{stable} if $\mu_H( F) < \mu_H( E)$ holds for every nonzero coherent subsheaf $0 \neq F \subset E$ with $\dim \supp( E/F) \geq n-1$.
\begin{proposition} \label{prop:stable}
  Assume given a nonsplit short exact sequence
  \begin{equation} \label{eq:extension}
    0 \longto E_1 \longto E \longto E_2 \longto 0
  \end{equation}
  of vector bundles over $X$. Let $H$ be an ample divisor class on $X$. Suppose that $E_1$ and $E_2$ are slope $H$-stable with
  \begin{equation} \label{eq:same_slope}
    \mu_H( E_1) = \mu_H( E_2).
  \end{equation}
  Let $D$ be a divisor class on $X$ such that
  \begin{equation} \label{eq:stable_side}
    c_1( E_1) \cdot D H^{n-2} < c_1( E_2) \cdot D H^{n-2}.
  \end{equation}
  Then $E$ is slope $(H + \epsilon D)$-stable for each sufficiently small $\epsilon > 0$.
\end{proposition}
\begin{proof}
  We follow \cite[Section 1]{schmitt}.
  Replacing $H$ by a large multiple if necessary, we may assume without loss of generality that $H + D$ is still ample. We have
  \begin{equation*}
    (H + \epsilon D)^{n-1} = \sum_{i=0}^{n-1} f_i( \epsilon) H^i (H + D)^{n+1-i}, \qquad f_i( \epsilon) := \binom{n-1}{i} (1 - \epsilon)^i \epsilon^{n+1-i}.
  \end{equation*}
  The functions $f_i$ are nonnegative and bounded on $[0, 1]$, and $f_0( \epsilon) \geq 1/2^{n-1}$ for $\epsilon \in [0, 1/2]$.
  Thus the previous lemma yields a constant $C'$ such that
  \begin{equation*}
    c_1( F) \cdot (H + \epsilon D)^{n-1} \leq c_1( F) \cdot (H/2)^{n-1} + C'  
  \end{equation*}
  holds for all $\epsilon \in [0, 1/2]$ and all coherent subsheaves $F \subseteq E$. We may of course assume $C' \geq 0$. Then the previous inequality implies
  \begin{equation*}
    \mu_{H + \epsilon D}( F) \leq 1/2^{n-1} \mu_H( F) + C'
  \end{equation*}
  as long as $F$ is nonzero. Hence we conclude that
  \begin{equation} \label{eq:stable}
    \mu_{H + \epsilon D}( F) < \mu_{H + \epsilon D}( E)
  \end{equation}
  holds for all $\epsilon \in [0, 1/2]$ and all nonzero coherent subsheaves $F \subseteq E$ with
  \begin{equation} \label{eq:unbounded}
    1/2^{n-1} \mu_H( F) < \min\limits_{0 \leq \epsilon \leq 1/2} \mu_{H + \epsilon D}( E) - C'.
  \end{equation}
  It remains to prove the inequality \eqref{eq:stable} for all sufficiently small $\epsilon > 0$ and all coherent subsheaves $0 \neq F \subseteq E$ with $\dim \sup( E/F) \geq n-1$ that do not satisfy \eqref{eq:unbounded}.
  This class of subsheaves $F$ is bounded according to \cite[Lemma 1.7.9]{huybrechts-lehn}, so only finitely many of the polynomials
  \begin{equation*}
    \xi_F( \epsilon) := \mu_{H + \epsilon D}( F) - \mu_{H + \epsilon D}( E) 
  \end{equation*}
  are different. Given such a subsheaf $F \subseteq E$, we put
  \begin{equation*}
    F_1 := F \cap E_1 \subseteq E_1 \qquad\text{and}\qquad F_2 := \frac{F + E_1}{E_1} \subseteq \frac{E}{E_1} = E_2.
  \end{equation*}
  Thus we obtain a short exact sequence of torsionfree coherent sheaves
  \begin{equation*}
    0 \longto F_1 \longto F \longto F_2 \longto 0
  \end{equation*}
  which is contained in the given short exact sequence \eqref{eq:extension}. The stability of $E_1$ and $E_2$ imply $\mu_H( F) < \mu_H( E)$, and hence $\xi_F( 0) < 0$, except for the following two cases:
  \begin{itemize}
   \item If $F_2 = 0$ and $\dim \supp( E_1/F_1) \leq n-2$, then $\xi_F = \xi_{E_1}$ follows; in particular, $\xi_F( 0) = 0$ and $\xi_F'( 0) < 0$ due to the assumptions \eqref{eq:same_slope} and \eqref{eq:stable_side}.
   \item If $F_1 = 0$ and $\dim \supp( E_2/F_2) \leq n-2$, then the composition $F_2 \cong F \subseteq E$ would extend, by Hartogs' theorem, to a morphism $E_2 \to E$, which would split the given sequence \eqref{eq:extension}.
    So this case is impossible.
  \end{itemize}
  In each possible case, we obtain $\xi_F( \epsilon) < 0$ for all sufficiently small $\epsilon > 0$.
  Since only finitely many polynomials $\xi_F$ need to be taken care of, we can make $\epsilon$ sufficiently small for all of them simultaneously.
\end{proof}

\section{The Case $c_1(E)=0$ and $X$ being a Calabi-Yau Threefold}
Let $X$ be a Calabi-Yau threefold, and let $H$ be an ample divisor class on $X$.
For application in string theory, the construction of $H$-stable holomorphic vector bundles $E$ with vanishing first Chern class over $X$ is of particular interest.

Let $D$ be a divisor class on $X$. For vector bundles $E$ over $X$, we use the standard notation
\begin{equation*}
  E( D) := E \otimes \calO_X( D).
\end{equation*}
Suppose that two $H$-stable holomorphic vector bundles $E_1$ and $E_2$ over $X$ with $c_1( E_{\nu}) = 0$ are already given.
Put $r_{\nu} := \rank( E_{\nu})$, $r_\nu':=r_\nu/\gcd(r_1,r_2)$, $r':=r_1'+r_2'$ and $r := r_1 + r_2$ . Then every extension
\begin{equation} \label{eq:extension_D}
  0 \longto E_1( r_2' D) \longto E \longto E_2( -r_1' D) \longto 0
\end{equation}
satisfies $\rank( E) = r$ and $c_1( E) = 0$. The remaining Chern classes of $E$ are
\begin{align}
  \label{eq:c2} c_2(E) & = -\frac{r_1 r_2'^2+r_2r_1'^2 }{2} D^2 + c_2( E_1) + c_2( E_2) \qquad\text{and}\\
  \label{eq:c3} c_3(E) & = \frac{r_1r_2'^3-r_2r_1'^3}{3} D^3 + 2 \big( r_1' c_2( E_2) - r_2' c_2( E_1) \big) \cdot D + c_3( E_1) + c_3( E_2).
\end{align}
If $D \cdot H^2 = 0$, then $E$ is $H$-semistable, but not $H$-stable; the idea is to replace $H$ by $H + \epsilon D'$ for another divisor class $D'$ on $X$.
Applying Proposition \ref{prop:stable} to the extension $E$ of $E_2( -r_1' D)$ by $E_1( r_2' D)$ and the divisor class $D'$,
we see that the vector bundle $E$ is $(H + \epsilon D')$-stable for each sufficiently small $\epsilon > 0$ if the following three conditions hold:
\begin{itemize}
 \item The extension \eqref{eq:extension_D} does not split.
 \item $D \cdot H^2 = 0$.
 \item $D' \cdot D \cdot H < 0$. 
\end{itemize}
Let us discuss the solvability of these conditions.

Clearly a divisor class $D'$ with $D' \cdot D \cdot H < 0$ can only exist if $D \cdot H \not\equiv 0$ numerically.
Assuming the latter, the following lemma shows that such a divisor class $D'$ always exists, and that we can actually take $D' := D$.
\begin{lemma} \label{lemma:negative}
  Let $H$ and $D$ be divisor classes on a projective threefold $X$. Suppose that $H$ is ample, that $D \cdot H \not\equiv 0$ numerically, and that $D \cdot H^2 = 0$. Then $D^2 \cdot H < 0$.
\end{lemma}
\begin{proof}
  Bertini's theorem allows us to find a smooth projective surface $S \subset X$ in a linear system $|mH|$ with $m \gg 0$.
  The assumptions on $H$ and $D$ imply that the divisor class $H|_S$ is ample on $S$, that $D|_S \not\equiv 0$ numerically, and that $D|_S \cdot H|_S = 0$.
  Thus the Hodge Index Theorem yields $(D|_S)^2 < 0$ on $S$, and hence $D^2 \cdot H < 0$.
\end{proof}
\begin{remark}
  Combining the lemma with formula \eqref{eq:c2}, we see that the curve class
  \begin{equation*}
    c_2( E_1) + c_2( E_2) - c_2( E)
  \end{equation*}
  always has negative degree with respect to $H$. In particular, this method cannot produce extension bundles $E$ for which this curve class vanishes, or is effective.
\end{remark}
The next question is: Given the two $H$-stable holomorphic vector bundles $E_1$, $E_2$ with vanishing first Chern class on the Calabi-Yau threefold $X$, and the divisor class $D$ on $X$,
do nonsplit extensions as in \eqref{eq:extension_D} exist? For that, we consider the alternating sum
\begin{equation*}
  \chi_D( E_2, E_1) := \sum_{i = 0}^3 (-1)^i \dim \Ext^i \big( E_2( -r_1' D), E_1( r_2' D) \big).
\end{equation*}
\begin{lemma}
  The Euler characteristic $\chi_D( E_2, E_1)$ is given by the formula
  \begin{equation} \label{eq:HRR}
    \chi_D( E_2, E_1) = \frac{r_1 r_2 r'^3}{6} D^3 + r' \big( \frac{r_1 r_2}{12} c_2( X) - r_2 c_2( E_1) - r_1 c_2( E_2) \big) \cdot D + \frac{r_2}{2} c_3( E_1) - \frac{r_1}{2} c_3( E_2).
  \end{equation}
  If $\chi_D( E_2, E_1) < 0$ and $D \cdot H^2 = 0$, then there are nonsplit extensions of $E_2( -r_1' D)$ by $E_1( r_2' D)$.
\end{lemma}
\begin{proof}
  The formula \eqref{eq:HRR} is a direct application of the Hirzebruch-Riemann-Roch Theorem, using that the Chern character of $E_2( -r_1' D)^* \otimes E_1( r_2' D)$ is the product of
  the three factors
  \begin{align*}
    \ch( E_2^*) & = r_2 - c_2( E_2) - \frac{1}{2} c_3( E_2),\\
    \ch( E_1) & = r_1 - c_2( E_1) + \frac{1}{2} c_3( E_1),\\
    \ch( \calO_X( r' D)) & = 1 + r'D + \frac{r'^2}{2} D^2 + \frac{r'^3}{6} D^3.
  \end{align*}
  If $D \cdot H^2 = 0$, then the two $H$-stable vector bundles $E_2( -r_1' D)$ and $E_1( r_2' D)$ have the same slope zero. If they were isomorphic, we would have $\chi_D( E_2, E_1) = 0$. Thus the summand
  \begin{equation*}
    (-1)^i \dim \Ext^i \big( E_2( -r_1' D), E_1( r_2' D) \big)
  \end{equation*}
  in $\chi_D( E_2, E_1)$ vanishes for $i = 0$, and by Serre duality also for $i = 3$. Hence the summand with $i = 1$ cannot vanish if $\chi_D( E_2, E_1) < 0$, which means $\Ext^1\big( E_2( -r_1' D), E_1( r_2' D) \big) \neq 0$.
\end{proof}
The next question concerns the existence, for a given ample divisor class $H$ on the Calabi-Yau threefold $X$, of a divisor class $D$ on $X$ with $D \cdot H^2 = 0$ and $D \cdot H \not\equiv 0$ numerically.

Note that $D$ cannot be proportional to $H$, so $h^{1,1}( X) > 1$ is clearly a necessary condition. It turns out to be also a sufficient condition:
\begin{lemma}
  Let $H$ be an ample divisor class on a Calabi-Yau threefold $X$ with $h^{1,1}( X) > 1$. Then there is a divisor class $D$ on $X$ with $D \cdot H^2 = 0$ and $D \cdot H \not\equiv 0$ numerically.
\end{lemma}
\begin{proof}
  The assumption $h^{1,1}( X) > 1$ implies by duality that there is a numerically nontrivial curve class $c$ on $X$ with $c \cdot H = 0$.
  The Lefschetz hyperplane theorem allows us to find such a curve class $c$ of the form $c = D \cdot H$ for a divisor class $D$ on $X$.
\end{proof}

Finally, let an ample divisor class $H$ on a Calabi-Yau threefold $X$ with $h^{1,1}( X) > 1$ be given, together with $H$-stable vector bundles $E_1$, $E_2$ with vanishing first Chern class on $X$.
Using the divisor class $D$ provided by the previous lemma, we get stable extensions $E$ of $E_2( -r_1' D)$ by $E_1( r_2' D)$ as in \eqref{eq:extension_D}
as soon as the nonsplit condition $\chi_D( E_2, E_1) < 0$ is satisfied. If this nonsplit condition fails, then we replace $D$ by a nonzero multiple $m D$ with $\chi_{m D}( E_2, E_1) < 0$.

It remains to consider the case where no nonzero integer $m$ with $\chi_{m D}( E_2, E_1) < 0$ exists.
This can only happen if $D^3 = 0$, since the first summand $\frac{r_1 r_2 r'^3}{6} (m D)^3$ dominates the expression for $\chi_{m D}( E_2, E_1)$ in \eqref{eq:HRR} if we choose $m \gg 0$ or $m \ll 0$.
In this quite exceptional case, one can for example replace $H$ by $H' := H + \delta D$ for a small rational number $\delta$,
and then use the previous lemma to obtain a divisor class $D'$ with $D' \cdot (H')^2 = 0$ and $D' \cdot H' \not\equiv 0$ numerically.
It is easy to check that this can be done in such a way that $(D')^3 \neq 0$. In this way, we can always construct stable extension of a twist of $E_2$ by a twist of $E_1$.

However, in this generality it is difficult to control the divisor class $D$, and hence in particular the second Chern class of the extension bundle $E$.
In some sense, this general existence result for stable extensions on Calabi-Yau threefolds $X$ with $h^{1,1}( X) > 1$ can be viewed as a more specific version of Maruyama's general existence result
\cite[Proposition A.1]{maruyama} for stable bundles on smooth projective varieties.

\section{Examples and Application to String Theory}

We end this note by providing two explicit examples of stable 
stable rank 2 and rank 4 bundles as extensions of two line bundles, or of the twisted tangent bundle by a line bundle, respectively. (In the above extension we specify in the rank 2 case $E_1=E_2={\mathcal O}_X$ and in the rank 4 case $E_1=TX$ and $E_2={\mathcal O}_X$.) 
This has the advantage that we can rely simply on given geometrical data associated to the Calabi-Yau geometry.
Another advantage is that in both cases the ample class $H$ is left unspecified.
So for this class of bundles we simply have to solve $D\cdot H^2=0$ assure that $D\cdot H$ is not numerically zero and solve the corresponding nonsplit conditions for
\begin{align*}
r&=2\colon\ \ \ c_2(X)D+8D^3<0,\\
r&=4\colon\ \ \ -3c_2(X)D+32D^3+\frac{c_3(X)}{2}<0.
\end{align*}
It is obvious that the constraints which involve $H$ force us to consider Calabi-Yau spaces with $h^{1,1}(X)>1$.

We now will analyze these constraints on two classes of Calabi-Yau manifolds, those which admit a K3 fibration or an elliptic fibration structure. These threefolds have been intensively studied in the physics literature as they provide examples for string compactifications. In particular for heterotic string compactifications it is important to specify a stable vector bundle $E$ with vanishing first Chern class and 
and second Chern class satisfying the constraint
\begin{equation}\label{ano}
c_2(X)-c_2(E)=[W],
\end{equation}
with $[W]$ an effective curve class in $X$. Condition (\ref{ano}) is the integrability condition for the existence of a solution of the heterotic anomaly equation. For the case $[W]=0$ it has been shown
that $X$ and $E$ can be deformed to a solution of the anomaly equation even already on the level of differential forms \cite{AG10}, \cite{AG102} (generalizing results of \cite{LY}, \cite{YK}). Thus it is of interest
to see if a given stable vector bundle satisfies (\ref{ano}) and so provides a solution to the basic 
consistency constraint imposed by heterotic string theory.

\begin{example}
Let ${X}$ be the Calabi-Yau threefold obtained as a blow-up
\begin{equation*}
  p\colon X \longto {\bar X}
\end{equation*}
of ${\bar X}$, which is a hypersurface of degree 8 in the weighted projective space ${\mathbb P}^4_{(1,1,2,2,2)}$, in the singular locus $C \subset {\bar X}$, which is a curve.
This Calabi-Yau manifold $X$ has been extensively studied in \cite{COFKM}. $X$ admits the structure of a $K3$-fibration
\begin{equation*}
  \pi\colon X \longto {\mathbb P}^1
\end{equation*}
with general fiber a quartic surface in ${\mathbb P}^3$.

Following \cite{COFKM}, we denote by ${\mathcal E}$ the exceptional divisor of $p$ and by ${\mathcal L}$ the fiber of $\pi$. The classes ${\mathcal L}$ and ${\mathcal E}$ generate $H_4(X,{\mathbb Z})$.
The degree two polynomials generate another linear system $|{\mathcal H}|$, which is related to the previous ones by $|{\mathcal H}|=|{\mathcal E}+2{\mathcal L}|$.

Let $l$ denote the fiber of ${\mathcal E}$ over $C$; then ${\mathcal H} \cdot {\mathcal E} = 4l$. Moreover, ${\mathcal H} \cdot {\mathcal L} = 4h$ for a curve class $h$. The classes $h$ and $l$ generate $H_2(X,{\mathbb Z})$.
The intersection numbers are given by (note ${\mathcal L}^2=0$)
\begin{equation*}
  {\mathcal E}^3=-16, \quad {\mathcal E}^2\cdot {\mathcal L}=4 , \quad {\mathcal E}\cdot h=1 , \quad {\mathcal E}\cdot l=-2 , \quad {\mathcal L}\cdot h=0 , \quad {\mathcal L}\cdot l=1.
\end{equation*}
Moreover, $h^{1,1}(X)=2$, $h^{2,1}(X)=86$ and 
$$c_1(X)=0, \ \ \ c_2(X)=56h+24l, \ \ \ c_3(X)=-168.$$
Using the basis $({\mathcal E}, {\mathcal L})$ of $H_4(X,{\mathbb Z})$ to describe the ample cone, a generic ample divisor class is written $H=s{\mathcal E}+t{\mathcal L}$ with $s,t\in {\mathbb R}^{>0}$ and $t>2s$.
Let $D=x{\mathcal E}+y{\mathcal L}$ be a divisor class; then
\begin{equation*}
  D\cdot H^2=4s\big[ys+2x(t-2s)\big].
\end{equation*}
Thus $D \cdot H^2 = 0$ is solved by $y=-\frac{2t-4s}{s}x$. Moreover $D \cdot H \cdot {\mathcal L} = 4 x s$, and hence $D \cdot H \not\equiv 0$ numerically as soon as $x \neq 0$.
The nonsplit conditions for rank $2$ and $4$ are given by
\begin{align*}
r&=2\colon \ \ \ x+3y+12x^2y-16x^3<0, \\
r&=4\colon \ \ \ -6x-18y+32x^2 (3y-4x) -21<0.
\end{align*}
We conclude with $y=-\frac{2t-4s}{s}x$ and $x>0$ both inequalities are solved and we get a class of stable rank $2$ and rank $4$ bundles on $X$.
\end{example}
We can now investigate the physical constraint (\ref{ano}). For this let us fix $(s,t)=(1,\frac{5}{2})$ thus 
$y=-x$. The second Chern class of the rank 2 is then given by
$$c_2(E)=4x^2(4h-l)$$
and computing the curve class $[W]$ in (\ref{ano}) we find
$$[W]= (56-16x^2)h+(24+4x^2)l$$
which is effective for $x=1$ only. In case of rank 4 bundles we get 
$$c_2(E)=c_2(TX)-6D^2,$$
therefore the curve class $[W]=6D^2$. To obtain a physical solution, we would need that $[W]$ is an effective curve class in $X$, but since $D^2 \cdot H < 0$ by Lemma \ref{lemma:negative}, we find that $[W]$ will not be effective.
\begin{example} Let $B$ be either a Hirzebruch surface ${\mathbb F}_n$ with $n=0,1,2$ or the del Pezzo surface $dP_k$ with $k=1,\dots, 8$. We set $c_i := c_i( B)$. Let
\begin{equation*}
  \pi\colon X \longto B
\end{equation*}
be an elliptically fibered Calabi-Yau threefold with section $\sigma$. The Chern classes of $X$ are
\begin{equation*}
  c_1(X)=0,\ \ \ c_2(X)=\pi^*c_2+11\pi^*c_1^2+12\sigma\cdot \pi^*c_1, \ \ \  c_3(X)=-60\sigma \cdot\pi^*c_1^2
\end{equation*}
according to \cite{FMW}. Let $\rho$ be an ample divisor class on $B$. Then
\begin{equation*}
  H = z \sigma + \pi^* \rho
\end{equation*}
is an ample divisor class on $X$ for $z>0$ and $\rho-zc_1$ ample on $B$. Let another divisor class
\begin{equation*}
  D = x \sigma + \pi^* \alpha
\end{equation*}
on $X$ be given. Using $\sigma^2 = -\sigma \cdot \pi^* c_1$, we get
\begin{equation*}
  D \cdot H^2 = x( \rho - z c_1)^2 + z( 2 \rho - z c_1) \cdot \alpha.
\end{equation*}
The nonsplit conditions for rank $2$ and rank $4$ are given by
\begin{align}
r&=2\colon \ \ x c_2 + ( 8x^3 - x) c_1^2 + 24x \alpha^2 + 12( 1 - 2x^2) \alpha \cdot c_1 < 0,\label{r2} \\
r&=4\colon \  -3x c_2 + ( 32x^3 + 3x - 30) c_1^2 + 96 x \alpha^2 - 12( 8x^2 + 3) \alpha \cdot c_1 < 0.\label{r4}
\end{align}
One class of solutions to (\ref{r2}) and (\ref{r4}) is obtained for $x<0$ and $\alpha$ effective. As we
assume that $h^{1,1}(X)>1$ we can always solve $D\cdot H^2=0$. 

To construct bundles which satisfy the physical constraint we are again restricted to rank 2 bundles.
As there are in principle numerous solutions, it is instructive to first see which possible restrictions
come from the physical anomaly equation. We get
$$[W]=\pi^*\big[(12-x^2)c_1+2x\alpha)\big]\cdot \sigma+\pi^*\big[c_2+11c_1^2+\alpha^2\big].$$ 
From $D\cdot H^2=0$ we get for $x\neq 0$ that $x\alpha<0$, we need $(12-x^2)c_1+2x\alpha\geq 0$
thus we find that $x^2=1,4,9$ are possible. For the last term we find $\alpha^2\geq -92$ (as $c_2=4$ and $c_1^2=8$ for rational $B$). For the class of solutions with $x<0$ and $\alpha$ effective we could take 
for instance $x=-1,-2$ and $\alpha=c_1$ solving $[W]\geq 0$. 
\end{example}

\end{document}